\documentclass[11pt,a4paper]{article}
\usepackage{amsmath,amssymb,amscd,amsthm}
\usepackage{comment}
\usepackage{enumerate}
\usepackage[all]{xy}

\newtheorem{lemma}{Lemma}

\newtheorem{prop}[lemma]{Proposition}
\newtheorem{cor}[lemma]{Corollary}

\newtheorem{defi}[lemma]{Definition}
\newtheorem{thm}[lemma]{Theorem}

\newtheorem{rmk}[lemma]{Remark}

\newcommand{\bb}[1]{\mathbb{#1}}
\newcommand{\fr}[1]{\mathfrak{#1}}

\newcommand{\ra}{\rightarrow}
\newcommand{\D}{\mathcal{V}}

\newcommand{\Der}{\mathrm{Der}}
\newcommand{\Sp}{\mathbb{S}}
\newcommand{\GL}{\mathrm{GL}}
\newcommand{\del}{\partial}
\newcommand{\Hom}{\mathrm{Hom}}

\DeclareMathOperator{\rank}{rank}

\begin{document}

\title{Representations of the Lie algebra of vector fields on a sphere}
\author{Yuly Billig and Jonathan Nilsson}
\date{}
\maketitle

\begin{abstract}
\noindent For an affine algebraic variety $X$ we study a category of modules that admit compatible actions of both the algebra of functions on $X$ and the Lie algebra of vector fields on $X$. In particular, for the case when $X$ is the sphere $\mathbb{S}^2$, we construct a set of simple modules that are finitely generated over $A$. In addition, we prove that the monoidal category that these modules generate is equivalent to the category of finite-dimensional rational $\mathrm{GL}_2$-modules.
\end{abstract}

\section{Introduction}

In 1986 David Jordan proved simplicity of Lie algebras of polynomial vector fields on smooth irreducible affine algebraic varieties (\cite{J1}, see also \cite{J2} and \cite{BF3}).
Yet, representation theory for this important class of Lie algebras remains largely undeveloped. The goal of the present paper is to investigate as a test case, representation theory of the Lie algebra of polynomial vector fields on a sphere $\Sp^2$. Previously, representation theory was developed only for the classical Lie algebras of Cartan type -- polynomial vector fields on an affine space and on a torus. Representations of the Lie algebras of vector fields on an affine space were studied by Rudakov in 1974 \cite{Ru}. 

An important classification result on representations of the Lie algebra of vector fields on a circle was established by Mathieu \cite{Ma}.
The case of an $N$-dimensional torus was studied by Larsson \cite{La}, Eswara Rao \cite{E1}, \cite{E2}, Billig \cite{Bi}, Mazorchuk-Zhao \cite{MZ} and Billig-Futorny \cite{BF1}. The culmination of this work was the proof in \cite{BF2} of Rao's conjecture on classification of irreducible weight modules for the Lie algebra of vector fields on a torus with finite-dimensional weight spaces. According to this classification, every such module is either of the highest weight type, or a quotient of a tensor module.

When we move from the torus to general affine algebraic varieties, the first difficulty that arises is the absence of a Cartan subalgebra. It was shown in \cite{BF3} that even in the case of an affine elliptic curve, the Lie algebra of vector fields does not contain non-zero semisimple or nilpotent elements. This demonstrates that the theory of this class of simple Lie algebras is very different from the classical theory of simple finite-dimensional Lie algebras, where roots and weights play a fundamental role. When studying representation theory of a simple Lie algebra, one has to impose some reasonable restrictions on the class of modules, since the description of simple modules in full generality is only known for $\mathfrak{sl}_2$ \cite{Bl}. 

In case of vector fields on a torus, a natural restriction is the existence of a weight decomposition with finite-dimensional weight spaces. A theorem proved in \cite{BF2} states that a simple weight module $M$ for the Lie algebra of vector fields on a torus admits a cover $\widehat{M} \rightarrow M$, where $\widehat{M}$ is a module for both the Lie algebra of vector fields and the commutative algebra of functions on a torus.
This suggests that a reasonable category of modules for the Lie algebra $\D$ of vector fields on an affine variety $X$ will be those that admit a compatible action of the algebra $A$ of polynomial functions on $X$. We refer to such modules as $A\D$-modules. A finiteness condition will be a requirement that the module is finitely generated over $A$.

We begin by discussing the general theory of $A\D$-modules, defining dual modules and tensor product in this category. Then the focus of the paper shifts to the study of the case of a sphere $\Sp^2$. A new feature here compared to the torus and circle is that Lie algebra of vector fields on $\Sp^2$, as well as its $A\D$-modules, are not free as modules over $A$. We construct a class of tensor modules (geometrically these are modules of tensor fields on a sphere), and prove their simplicity in the category of $A\D$-modules. We also show that the monoidal category generated by simple tensor modules on $\Sp^2$ is semisimple and is equivalent to the category of finite-dimensional rational $\GL_2$-modules. 

The methods that we employ are a combination of Lie theory and commutative algebra. Hilbert's Nullstellensatz is an essential ingredient in the proof of simplicity of tensor modules.  

\section{Generalities}
Let $X \subset \bb{A}^{n}$ be an algebraic variety over an algebraically closed field ${\bf k}$ of characteristic zero. Write $A_X$ for the algebra of polynomial functions on $X$, and let $\D_X=\Der_{\bf k}(A_X)$ be the Lie algebra of polynomial vector fields on $X$. When the variety $X$ is understood by the context we shall drop it as a subscript. Note that $\D$ is an $A$-module and that $A$ is a $\D$-module. For $a\in A$ and $\eta \in \D$ we shall write the latter action as $\eta(a)$.

Consider a vector space $M$ equipped with a module structure for both the associative commutative unital algebra $A$ and for the Lie algebra $\D$, such that the two structures are compatible in the following sense: 
For every $a \in A$, $\eta \in \D$, and $m \in M$ we have
\[\eta \cdot (a \cdot m) = \eta(a) \cdot m + a \cdot (\eta \cdot m).\]
This is equivalent to saying that $M$ is a module over the smash product algebra $A\#U(\D)$, see \cite{Mo} for details. For simplicity we shall write just$A\D$ for $A\#U(\D)$. A morphism of $A\D$-modules is a map that preserves both the $A$- and the $\D$-structures.
The category of $A\D$-modules is clearly abelian since $A$-Mod and $\D$-Mod are.

For two $A\D$-modules $M$ and $N$ we may form the tensor product $M\otimes_{A}N$ which makes sense since $A$ is commutative. This is equipped with the natural $A$-module structure
\[a\cdot(m\otimes n) = (a\cdot m)\otimes n,\]
where the right side also equals $m \otimes (a\cdot n)$.

The Lie algebra of vector fields acts on the tensor product by
\[\eta(m\otimes n) = (\eta \cdot m)\otimes n + m\otimes (\eta \cdot n),\]
as usual. We verify that these two structures are compatible in the above sense. We have 
\begin{align*}
\eta \cdot (a\cdot (m\otimes n)) &= \eta \cdot ((a\cdot m)\otimes n)=(\eta \cdot(a\cdot m))\otimes n + (a\cdot m)\otimes (\eta \cdot n) \\
&= (\eta(a) \cdot m) \otimes n +(a\cdot (\eta \cdot m))\otimes n + a\cdot (m \otimes (\eta \cdot n)) \\
&=\eta(a) \cdot(m \otimes n) + a\cdot(\eta \cdot (m\otimes n)).
\end{align*}
This shows that $M\otimes_{A}N$ is an $A\D$-module, and that $A\D$-Mod is a monoidal category.

For any $A\D$-module $M$ we define 
\[M^{\circ}:= \Hom_{A}(M,A).\]
The algebra $A$ acts naturally on $M^{\circ}$  by $(a\cdot \varphi)(m)=a\varphi(m)$, and we define an action of $\D$ on $M^{\circ}$ by
\[(\eta \cdot \varphi)(m) = -\varphi(\eta \cdot m)+\eta(\varphi(m))\]
for all $\eta \in \D$, $\varphi\in M^{\circ}$, and $m \in M$. These two actions are compatible in the sense defined above, so $M^{\circ}$ is indeed an $A\D$-module.
The contravariant functor $M \mapsto M^{\circ}$ provides a duality on $A\D$-Mod.

\subsection*{Chart Parameters}
Let $h\in A$ be a function and consider the corresponding chart for $X$ consisting of points where $h$ does not vanish: $N(h)=\{p\in X | h(p) \neq 0\}$.

\begin{defi}
\label{chartdef}
We shall say that $t_1, \ldots, t_s \in A$ are {\bf chart parameters} in the chart $N(h)$ provided that the following conditions are satisfied:
\begin{enumerate}
	\item $t_1, \ldots, t_s$ are algebraically independent over ${\bf k}$, so ${\bf k}[t_1, \ldots, t_s] \subset A$.
	\item Each element of $A$ is algebraic over ${\bf k}[t_1, \ldots, t_s]$.
	\item For each index $p$, the derivation $\frac{\del}{\del t_p} \in \Der({\bf k}[t_1,\ldots, t_s])$ extends to a derivation of the localized algebra $A_{(h)}$.
\end{enumerate}
\end{defi}

Note that part $2$ also implies that each element of $A_{(h)}$ is algebraic over ${\bf k}[t_1,\ldots, t_s]$ since algebraic elements are closed under taking inverses. Some further consequences of the definition are given below.

\begin{lemma}
Let $t_1, \ldots, t_s \in A$ be chart parameters in the chart $N(h)$. Then
\begin{enumerate}
	\item The extension of the derivation $\frac{\del}{\del t_p} \in \Der({\bf k}[t_1,\ldots, t_s])$ to the localized algebra $A_{(h)}$ is {\bf unique}.
	\item $\Der(A_{(h)}) = \displaystyle\bigoplus_{p=1}^{s} A_{(h)} \frac{\del}{\del t_p}$. 
\end{enumerate}
\end{lemma}
\begin{proof}
The uniqueness in the first claim follows from part $2$ of Definition~\ref{chartdef}. Indeed let $f$ be a non-zero element of $A_{(h)}$ and let 
$p_nT^n+\cdots + p_1T+p_0$ be its minimal polynomial with $p_i\in {\bf k}[t_1,\ldots, t_s]$. Let $\mu \in \Der(A_{(h)})$. Then 
\[\mu (f) = -\frac{\mu(p_n)f^n + \cdots + \mu(p_1)f+\mu(p_0)}{np_nf^{n-1}+ \cdots + p_1},\]
hence every derivation of $A_{(h)}$ is uniquely determined by its values on \break ${\bf k}[t_1, \ldots, t_s]$.

The second claim follows from part $3$ of Definition~\ref{chartdef}: For any $d\in \Der(A_{(h)})$, let $d'=\sum_{i=1}^s d(t_i)\frac{\del}{\del t_i}$. Then $d$ and $d'$ are both derivations of $A_{(h)}$ which are equal on ${\bf k}[t_1, \ldots, t_s] \subset A_{(h)}$, hence $d=d'$. Moreover, the expression of a derivation as an $A_{(h)}$-combination of $\{\frac{\del}{\del t_1}, \ldots, \frac{\del}{\del t_s}\}$ is unique; this is seen when applying such a combination to $t_1, \ldots , t_s$.
\end{proof}

Our prototypical example is the following. Let $X=\Sp^2$ and let $h=z$. Then $x,y$ are chart parameters outside the equator $z=0$: First ${\bf k}[x,y] \subset A$. Next we have $q(z)=0$ for $q(T)=T^2+(x^2+y^2-1)$ so $z$ is algebraic over ${\bf k}[x,y]$. Since $x,y,z$ generate the ring $A$, every element of $A$ is algebraic over ${\bf k}[x,y]$. For the third part we note that $0=\frac{\del}{\del x}(x^2+y^2+z^2-1) = 2x+2z\frac{\del z}{\del x}$, so for the extension of $\frac{\del }{\del x}$ to $A_{(z)}$ we have $\frac{\del z}{\del x}=-\frac{x}{z}$ which uniquely determines the derivation $\frac{\del}{\del x}$ on $A_{(z)}$. By symmetry the same goes for $\frac{\del}{\del y}$.

%Include?
\begin{comment}
\begin{lemma}
Suppose that $t_1, \ldots, t_s$ are chart parameters in the chart $N(h)$. Fix a point $P \in N(h)$ and define $\tilde{t}_i:=t_i-t_i(P)$.
Then $\tilde{t}_1, \ldots, \tilde{t}_s$ are local parameters at a point $P$ in the classical sense, namely $\tilde{t}_1, \ldots, \tilde{t}_s \in \fr{m}_{p}$ and their images form a basis for the vector space $\fr{m}_{p} / \fr{m}_{p}^{2}$.
\end{lemma}
\end{comment}

\subsection*{An Atlas for $X$}
Let $X$ be the zero locus of polynomials $g_1, \ldots, g_n \in {\bf k}[x_1, \ldots, x_m]$ such that $A={\bf k}[x_1, \ldots, x_m]/\langle g_1,\ldots, g_n \rangle$, and define \[J=\left( \frac{\del g_i}{\del x_j}\right) \in \mathrm{Mat}_{n\times m}(A).\]
Let $F$ be the field of fractions of $A$ and define $r:=\rank_F J$. This means that there exists some $r\times r$-minor which is a nonzero element of $A$. By the Nullstellensatz we have $\rank_F J = \max_{P\in X} \rank_{\bf k} J(P)$, and when $X$ is smooth, $\rank_{\bf k}J(P)$ is independent of the point $P$ (see \cite{S} Section 2.1.4).

We now consider $r\times r$-minors of $J$. For $\alpha \subset \{1,\ldots, n\}$ and $\beta \subset \{1,\ldots, m\}$ we write $J^{\alpha,\beta}$ for the corresponding $r\times r$-minor of $J$. For $h\in A$, let $N(h)=\{P\in X \; | \; h(P) \neq 0\}$.
\begin{lemma}
Let $X$ be smooth and let $r=\rank_F J$.  Then the following set of charts forms an atlas for $X$:
\[\{ N( \det J^{\alpha,\beta}) \; \big| \; |\alpha|=|\beta|=r, \; \det(J^{\alpha,\beta})\neq 0 \}.\]
\end{lemma}
\begin{proof}
Let $P\in X$. Since $X$ is smooth, $\rank_{\bf k} J(P)=\rank_F J = r$. Thus there exists a nonzero minor in $J(P)$: $\det(J^{\alpha,\beta}(P)) \neq 0$ for some $\alpha$ and $\beta$ of order $r$, so $P \in N( \det J^{\alpha,\beta})$. Thus the above atlas covers $X$.  
\end{proof}
From here on we shall fix this atlas for the variety $X$.

For the sphere $X=\Sp^2$, we shall sometimes write $(x,y,z)$ for $(x_1,x_2,x_3)$. We have $A=k[x,y,z] / \langle g \rangle$ with $g=x^2+y^2+z^2-1$, so $J=(2x \;\; 2y \;\; 2z)$. Here $\rank_F J =1$ and there are three nonzero minors: $2x,2y$, and $2z$. So in this case our charts are $N(x)$, $N(y)$, and $N(z)$ - each obtained by removing a great circle from the sphere.

As described in \cite{BF3}, the matrix $J$ also grants an explicit description of the vector fields on $X$: for $f_i \in A$,  the combination $\sum_{i=1}^m f_i\frac{\del}{\del x_i}$ is a vector field on $X$ if and only if the vector $(f_1,\ldots, f_m)$ belongs to the kernel of $J$. 

\begin{lemma}
Let $r=\rank_F J$ and let $J^{\alpha,\beta}$ be a minor of size $r$ with $h=\det J^{\alpha,\beta} \neq 0$. Then $\{x_i \; |\; i\not\in \beta\}$ are chart parameters in the chart $N(h)$.
\end{lemma}
\begin{proof}
Consider $\Der_{\bf k} F$. It is easy to see that for any $s\in F$ and $\eta\in \Der(A)$ we have $s\eta \in \Der_{\bf k} F$ and conversely, for any $\mu \in \Der_{\bf k}F$ there exists $q\in A$ such that $q\mu \in \Der_{\bf k}(A)$. Then derivations of $F$ can be written as $\sum_{i=1}^m f_i\frac{\del}{\del x_i}$ where $(f_1,\ldots, f_m)$ are solutions over $F$ of a system of linear homogeneous equations with matrix $J$. Since $\rank_{F}J=r$, we can keep only the rows $\alpha$ and choose variables $\{f_i \; |\; i\in \beta\}$ as leading, and $\{f_i \; | \; i \not\in \beta\}$ as free. Writing down the fundamental solutions, we get derivations $\tau_j = \frac{\del}{\del x_j} + \sum_{i \in \beta} f_{i,j} \frac{\del}{\del x_i} \in \Der_{\bf k} F$ for each $j\not\in \beta$. Note that only $h$ may appear in the denominators of $f_{i,j}$, hence $h\tau_j \in \Der_{\bf k}(A)$ for all $j \not\in \beta$. This implies that $\{x_i \; |\; i\not\in \beta\}$ are algebraically independent. Indeed if $p$ is a polynomial in $\{x_i \; |\; i\not\in \beta\}$ which vanishes in $F$, we can apply a sequence of derivations $\tau_j$ which brings $p$ to $1$ and obtain a contradiction $1=0$. If all free variables have zero values, the solution of a homogeneous system is trivial. This implies that every derivation of $F$ which is zero on ${\bf k}[x_i \; | \; i\not\in \beta]$, is zero on $F$.
 By~\cite{Lang} Chapter VIII, Prop. 5.2, $F$ is algebraic over ${\bf k}(x_i \; | \; i\not\in \beta)$.
\end{proof}
\begin{rmk}
In the above lemma, we may assume without loss of generality that $\{i \; | \; i \not\in \beta\}=\{1,\ldots,s\}$
and $\{i \; | \; i \in \beta\} = \{s+1,\ldots , n\}$.

If we treat the chart parameters $t_i=x_1, \ldots , t_s=x_s$ as independent variables and $x_{s+1}, \ldots, x_n$ as dependent we can write a derivation
\[\sum_{i \not\in \beta}f_i \frac{\del}{\del x_i}+\sum_{j \in \beta}f_j \frac{\del}{\del x_j}\]
 simply as $\sum_{i=1}^s f_i \frac{\del}{\del t_i}$ with understanding that for $j\in \beta$ we have \[\sum_{i=1}^s f_i \frac{\del x_j}{\del t_i} = \sum_{i=1}^s f_i \tau_i(x_j).\]
\end{rmk}

\subsection*{Embedding of Vector Fields}
The embedding $A \subset A_{(h)}$ gives a corresponding embedding of polynomial vector fields: 
\[\mathrm{Vect}(X) \simeq \Der(A) \subset \Der(A_{(h)})=\bigoplus_{p=1}^{n} A_{(h)} \frac{\del}{\del t_p},\]
In other words, a polynomial vector field on $X$ can be written as
$\sum_{i=1}^{n}f_i\frac{\del}{\del t_i} \in \Der(A_{(h)})$ for some unique $f_i \in A_{(h)}$.

\subsection*{Valuation}
For each point $P\in N(h)$, define $\nu_P: A_{(h)}\setminus\{0\} \ra \bb{N}$ by
\[\nu_P(f)=\min \left\{\sum_{i=s}^n \alpha_i\; \Bigg| \left(\prod_{i=1}^{s}\left(\tfrac{\del}{\del t_i}\right)^{\alpha_i} f \right)(P) \neq 0\; \right\}.\]
Then $\nu_P$ is well-defined by the following lemma.

\begin{lemma}
We have $\nu_P(f) < \infty$ for each nonzero $f\in A_{(h)}$.
\end{lemma}
\begin{proof}
First of all, $\nu_P$ is finite on ${\bf k}[t_1,\ldots, t_s] \subset A_{(h)}$ as it is bounded by the degree function. Now any $f\in A_{(h)}$ is by definition algebraic over ${\bf k}[t_1,\ldots, t_s]$ so we can consider its minimal polynomial $m(X)=\sum_{i=0}^{N}a_{i}X^k$, where $a_i\in {\bf k}[t_1,\ldots, t_s]$ and $m(f)=0$ and $a_0\neq 0$. We claim that $\nu_P(f) \leq \nu_P(a_0)$. To prove this, suppose for contradiction that $\left(\prod_{i=1}^{s}(\tfrac{\del}{\del t_i})^{\alpha_i} f \right)(P)=0$ for all $\alpha_1+\cdots+\alpha_s \leq \nu_P(a_0)$. Pick $\alpha_i$'s such that
 $\left(\prod_{i=1}^{s}(\tfrac{\del}{\del t_i})^{\alpha_i} a_0 \right)(P)\neq 0$. Then for any derivation $d=\frac{\del}{\del t_p}$ we have
\[d(a_0)=-\left( \sum_{i=1}^N d(a_i)f^i + a_iif^{i-1}d(f) \right).\]
But by assumption, both $f(P)=0$ and $d(f)(P)=0$ so the right hand side is zero at $P$, hence also $d(a_0)(P)=0$. Iterating this $\nu_P(a_0)$ times we get
$\left(\prod_{i=1}^{s}(\tfrac{\del}{\del t_i})^{\alpha_i} a_0 \right)(P)= 0$, a contradiction.
\end{proof}

Note that the above proof uses our assumption that $\mathrm{char} \; {\bf k} =0$.

\section{A class of $A\D$-modules}
Let $s:=\dim X$. Let $\{t_1, \ldots, t_s\}$ be chart parameters in a chart $N(h)$ and let $A_{(h)}$ be the localization of the algebra $A$ at $h$. 
For any $\fr{gl}_{s}$-module $U$ we consider the space $A_{(h)} \otimes_{\bf k} U$. 
The algebra $A$ acts on this space by multiplication on the left factor.

The proof of the following lemma is straightforward and we leave it to the reader.
\begin{lemma}
Let $U$ be a finite-dimensional $\fr{gl}_s$-module. Consider the vector fields $\D$ as embedded in $\Der(A_{(h)})=\bigoplus_{i=1}^{s}A_{(h)}\frac{\del}{\del t_{i}}$. Define an action of $\D$ on $A_{(h)} \otimes U$ by
\[\tag{1} \left(\sum_{i=1}^{s}f_i \frac{\del}{\del t_i}\right)\cdot (g\otimes u):=\sum_{i=1}^{s}f_i \frac{\del g}{\del t_i}\otimes u + \sum_{p=1}^{s}\sum_{i=1}^{s}g \frac{\del f_i}{\del t_p}\otimes (E_{p,i}\cdot u).\]
Here $E_{p,i}$ is a standard basis element of $\fr{gl}_s$, $g \in A_{(h)}$, and $u\in U$. This equips $A_{(h)} \otimes U$ with the structure of an $A\D$-module.
\end{lemma}

The algebra $A_{(h)}$ has a natural doubly infinite filtration with respect to powers of $h$:
\[\cdots \subset h^{k+1}A \subset h^kA\subset h^{k-1}A \subset \cdots.\]
%The corresponding graded algebra is $gr\: A_{(h)} = \bigoplus_{k \in \bb{Z}} h^kA / h^{k+1}A$. 
%Note that the canonical projection $\pi: A_{(h)} \ra gr \: A_{(h)}$ is multiplicative but not additive. 
%where $h^kA / h^{k+1}A \simeq {\bf k}[x,y] / \langle x^2+y^2-1 \rangle$ for each $k$.
For each $a\in A_{(h)}$ we define its {\bf degree} by
\[\deg a := \max \{k\in \bb{Z}\; | \; a \in h^{k}A\}.\]

We extend our notion of degree to elements of $A_{(h)} \otimes U$ in the natural way: for nonzero $m\in A_{(h)} \otimes U$ we define
\[\deg m := \max \{k\in \bb{Z}\; | \; m \in h^{k}A \otimes U\}.\]
If $M \subset A_{(h)} \otimes U$ is a nonzero submodule, we define
\[\deg(M):= \inf \{\deg m | \; m \in M,\; m\neq 0\}.\]

Let $M \subset A_{(h)} \otimes U$ be a submodule. We shall call $M$ {\bf bounded} if  $\deg(M)$ is finite. It is easy to see that if $M$ is finitely generated over $A$ then it is bounded. Conversely, every bounded submodule in $A_{(h)}\otimes U$ is finitely generated over $A$. Indeed, as an $A$-module $h^kA\otimes U$ is isomorphic to $A\otimes U$ and since $A$ is noetherian, every submodule in a finitely generated $A$-module is finitely generated.

 On the other hand, we shall call $M$ {\bf dense} if $M \supset h^{k}A\otimes U$ for some $k$.
Note that $M$ is both dense and bounded when there exist two integers $K\geq k$ such that
\[ h^{K}A\otimes U \subset M \subset h^{k}A\otimes U.\]

\section{The Sphere}
From here on we shall focus on the case when $X$ is the sphere $\Sp^2$. Some results still hold in a more general setting.

Let $X=\Sp^{2} \subset \bb{A}^{3}$. With notation as above we have $A={\bf k}[x_1,x_2,x_3] / \langle x_1^2 + x_2^2 + x_3^2-1 \rangle$. However, we shall sometimes write $x,y,z$ for $x_1,x_2,x_3$.
Let $\Delta_{ij} = x_{j}\frac{\del}{\del x_i} - x_{i}\frac{\del}{\del x_j}$. Then it is easy to check that $\D$ is generated by $\Delta_{12}$, $\Delta_{23}$, and $\Delta_{31}$ as an $A$-module, and that these generators satisfy
\[[\Delta_{12},\Delta_{23}] = \Delta_{31}, \qquad[\Delta_{23},\Delta_{31}] = \Delta_{12}, \qquad[\Delta_{31},\Delta_{12}] = \Delta_{23}.\] However, $\D$ is not a free $A$-module since we have the relation $x_{1}\Delta_{23}+x_{2}\Delta_{31}+x_{3}\Delta_{12} =0$.

In the chart $N(z)$ with chart parameters $\{x,y\}$ these generating vector fields are expressed as
\[\Delta_{12}=y\frac{\del}{\del x} - x\frac{\del}{\del y}, \qquad \Delta_{23} = z\frac{\del}{\del y}, \qquad \Delta_{31}=-z\frac{\del}{\del x}.\]

\section{Explicit construction of modules}
The Lie algebra $\fr{sl}_2$ acts naturally on ${\bf k}[X,Y]$ by
\[E_{1,2}\cdot f = X\frac{\del f}{\del Y}, \quad E_{2,1}\cdot f = Y\frac{\del f}{\del X}, \quad (E_{1,1}-E_{2,2})\cdot f=X\frac{\del f}{\del X}-Y\frac{\del f}{\del Y}.\]
For any $\alpha\in{\bf k}$ we may extend this to a $\fr{gl}_2$-module ${\bf k}_{\alpha}[X,Y]$ by requiring $(E_{1,1}+E_{2,2})\cdot f=\alpha f$.

Homogeneous components are preserved by this action, and ${\bf k}_{\alpha}[X,Y] = \bigoplus_{m\geq 0} U_m^{\alpha}$ where $U_m^{\alpha}$ is the homogeneous component of degree $m$.
Explicitly, for $m\in \bb{N}$ and for $0 \leq i \leq m$, we define $v_{i}^{m}:= {m \choose i} X^iY^{m-i}$ (we shall drop the upper $m$ when it is understood by the context).
Then $U_{m}^{\alpha}=\mathrm{span}\{v_0^m, \ldots, v_m^m\}$ is the homogeneous component of degree $m$, and the action on these basis elements is given by

\[E_{1,1} \cdot v_i = \frac{1}{2}(\alpha+m-2i)v_i, \qquad \qquad E_{1,2} \cdot v_i = (m-i+1)v_{i-1},\]
\[E_{2,1} \cdot v_i = (i+1)v_{i+1}, \qquad \qquad E_{2,2} \cdot v_i = \frac{1}{2}(\alpha-m+2i)v_{i}.\]
In particular this implies that
\[(E_{1,1}-E_{2,2})\cdot v_i = (m-2i)v_i, \qquad \qquad (E_{1,1}+E_{2,2})\cdot v_i = \alpha v_i.\]
Here $v_{-1}$ and $v_{m+1}$ are $0$ by definition. When restricted to $\fr{sl}_2$, $U_{m}^{\alpha}$ is the unique simple module of dimension $m+1$. In particular, any finite-dimensional simple $\fr{gl}_2$-module is isomorphic to $U_{m}^{\alpha}$ for a unique $\alpha$ and $m$.

\subsection*{Modules of rank $1$}
We first consider the case $m=0$. Here $U_{0}^{\alpha}$ is one dimensional and the identity acts by $\alpha$.
We consider the chart $N(z)$ on $\Sp^2$ with chart parameters $\{x,y\}$.
\begin{prop}
The module $A_{(z)} \otimes U_{0}^{\alpha}$ contains a bounded $A\D$-submodule if and only if $\alpha \in 2\bb{Z}$.
\end{prop}
\begin{proof}
Let $M$ be a bounded submodule of $A_{(z)}\otimes U_{0}^{\alpha}$. 
Let $w$ be a non-zero element of $M$ of lowest possible degree $k$. Then $w$ can be expressed
as $w=\sum_{i\geq k} z^ia_i\otimes v_0$ where $a_k\neq 0$. We compute
\begin{align*}
\Delta_{2,3} (z^ka_k \otimes v_0) &=z\frac{\del}{\del y}(z^ka_k)\otimes v_0 + z^ka_k\frac{\del z}{\del y}\otimes E_{2,2}v_0\\
&=z\left(z^k\frac{\del a_k}{\del y} - ka_k z^{k-1}\frac{y}{z}\right)\otimes v_0 - z^ka_k\tfrac{y}{z}\otimes\tfrac{\alpha}{2}v_0,
\end{align*}
which modulo the space $z^kA\otimes U_{0}^{\alpha}$ equals
\[-(k+\tfrac{\alpha}{2})a_kyz^{k-1}\otimes v_0.\]
So by the minimality of $k$ we must have $\alpha=-2k$. On the other hand it is easy to check that for $\alpha=2k$ the space $A^{\alpha} =z^k \otimes U_{0}^{\alpha}$ is an $A\D$-submodule in $A_{(z)} \otimes U_{0}^{\alpha}$. 
\end{proof}

\subsection*{Higher rank}
\subsubsection*{Module of $1$-forms on $X$}
In this section we consider for a moment an arbitrary $s$-dimensional variety $X$.
The space of $1$-forms $\Omega$ is an $A\D$-module where $A$ acts by left multiplication and vector fields act as follows:
\[\left(\sum_{i=1}^sf_i\frac{\del}{\del t_i}\right)\cdot \sum_{j=1}^s(g_j dt_j) = \sum_{i,j=1}^sf_i\frac{\del g_j}{\del t_i}dt_j
+g_jd\left(f_i\frac{\del}{\del t_i} t_j\right)\]
\[=\sum_{i,j=1}^sf_i\frac{\del g_j}{\del t_i}dt_j+\delta_{ij}g_jd(f_i) = f_i\frac{\del g_j}{\del t_i}dt_j+\delta_{ij}g_j \sum_{p=1}^{n} \frac{\del f_i}{\del t_p} dt_p.\]
By identifying $e_{i} \leftrightarrow dt_i$ we see that $\Omega \subset A_{(h)}\otimes V$, where the action on $V$ now is
$E_{p,i}e_j=\delta_{i,j}e_p$ which shows that $V$ is the natural $\fr{gl}_s$-module.

For $X=\Sp^2$ this means that $\Omega \subset A_{(z)}\otimes U_{1}^{1}$. In this case, the submodule $\Omega$ is generated by $dz=-z^{-1}xdx-z^{-1}ydy$.

\subsubsection*{Module of vector fields on $X$}
Similarly, the Lie algebra $\D$ of vector fields themselves forms an $A\D$-module in a natural way: $A$ acts by left multiplication, and $\D$ acts adjointly. We may rewrite this $\D$-action in the following way:
\[\left(\sum_{i=1}^sf_i\frac{\del}{\del t_i}\right)\cdot \sum_{j=1}^sg_j\frac{\del}{\del t_j} = \sum_{i,j=1}^sf_i\frac{\del g_j}{\del t_i}\frac{\del }{\del t_j}
-g_j\frac{\del f_i}{\del t_j}\frac{\del }{\del t_i}\]
\[=\sum_{i,j=1}^s\left(f_i\frac{\del g_j}{\del t_i}\frac{\del }{\del t_j}-\sum_{p=1}^{s}g_j\frac{\del f_i}{\del t_p} \cdot \delta_{p,j}\frac{\del}{\del t_i}\right).\]

Comparing with the definition $(1)$, we see that $\D$ is isomorphic to a submodule of $A_{(h)}\otimes U$ with
 $v_{i} \leftrightarrow \frac{\del}{\del t_i}$. The action on this module $\fr{gl}_s$-module $U$ is seen to be given by $E_{p,i} \cdot v_j = -\delta_{p,j} v_i$, which shows that $U$ is the dual of the natural $\fr{gl}_s$-module.

For $X=\Sp^2$ this means that $\D \subset A_{(z)}\otimes U_{1}^{-1}$. This $A\D$-submodule is generated by
 $\Delta_{1,2}=y\frac{\del}{\del x}-x\frac{\del}{\del y}$.

\begin{prop}
\label{closedop}
Let $M$ be an $A\D_{\Sp^2}$-submodule of $A_{(z)} \otimes U$, where $U$ is a finite-dimensional $\fr{gl}_2$-module. Then for $\sum_k g_k \otimes u_k \in M$ we also have $\sum_k (zg_k\otimes E_{i,j}\cdot u_k)\in M$ for all $1 \leq i,j \leq 2$. In other words, $M$ is closed under the operators $z\otimes E_{i,j}$.
\end{prop}
\begin{proof}
It suffices to prove the statement for a single term $g\otimes u$. For each vector field $\mu\in \D$ and for each function $f\in A$ we have
\[(f\mu)\cdot (g\otimes u)-f(\mu \cdot (g\otimes u)) \in M.\]
Taking $\mu=\Delta_{2,3}$ we obtain the following element in $M$:
\[f z\frac{\del g}{\del y} \otimes u + fg\frac{\del z}{\del x}\otimes E_{1,2}u + zg\frac{\del f}{\del x}\otimes E_{1,2}u + fg\frac{\del z}{\del y}\otimes E_{2,2}u+ zg\frac{\del f}{\del y}\otimes E_{2,2}u\]
\[-zf\frac{\del g}{\del y} \otimes u - fg\frac{\del z}{\del x}\otimes E_{1,2}u - fg\frac{\del z}{\del y}\otimes E_{2,2}u\]
\[=zg\frac{\del f}{\del x}\otimes E_{1,2}u+zg\frac{\del f}{\del y}\otimes E_{2,2}u.\]
Taking $f=x$ we obtain $zg\otimes E_{1,2} u \in M$. If we instead take $f=y$ we obtain $zg\otimes E_{2,2}u$.
Analogously, by taking $\mu=\Delta_{3,1}$ and $f=x,y$ we obtain $zg\otimes E_{1,1}u\in M$ and $zg\otimes E_{2,1}u\in M$.
\end{proof}

In what follows we shall use the following version of Hilbert's Nullstellensatz, see~\cite[Section 1.2.2]{S}.

\begin{lemma}
Let $I \lhd A$ be an ideal. Suppose that $g\in A$ satisfies $g(P)=0$ at all points $P\in X$ for which $f(P)=0$ for all $f\in I$. Then $g^k \in I$.
\end{lemma}

\begin{prop}
\label{density}
Let $U$ be a finite-dimensional irreducible $\fr{gl}_2$-module. Then every nonzero $A\D_{\Sp^2}$-submodule of $A_{(z)} \otimes U$ is dense.
\end{prop}
\begin{proof}
Since $U$ is simple and finite-dimensional it has form $U_m^{\alpha}$ as above.
Let $M \subset A_{(z)} \otimes U$ be a submodule and define
 \[I=\{f \in A | f (A \otimes U) \subset M\}.\]
Then $I$ is an ideal of $A$. To show that $M$ is dense we need to show that $z^N\in I$ for some $N$.

Let $v\in M$ and express this element in the form $v=\sum_{i=0}^{m}f_i\otimes v_i$, with $f_i \in A_{(z)}$. In fact we may assume that $f_i \in A$ (otherwise just multiply by a power of $z$).
By Lemma~\ref{closedop}, $M$ is closed under the operator $z\otimes E_{1,2}$, so we obtain $z^kf_0 \otimes v_0\in M$ for some $k$ and for some nonzero $f_0$. Acting by $z\otimes E_{2,1}$ repeatedly on this element we get $z^{k+i}f_0\otimes v_i$, so in particular we have $z^{N}f_0 \otimes U \subset M$, which shows that
 $z^{N}f_0 \in I$ so $I$ is non-zero. 
 
We now aim to apply Hilbert's Nullstellensatz to the function $g=z$. Fix $P\in \Sp^2$ with nonzero $z$-coordinate. We need to show that there exists $f\in I$ with $f(P)\neq 0$. We had already found $z^Nf_0\in I$ so if $f_0(P)\neq 0$ we are done. Otherwise, consider the element
 $z\frac{\del}{\del x} (z^{N+1}f_0\otimes v_0) \in M$. This expands as
\[z^{N+2}\frac{\del f_0}{\del x} \otimes v_0 - z^Nf_0\left((N+1)x\otimes v_0 +x\otimes E_{1,1}v_0-y\otimes E_{2,1}v_0 \right),\]
and since the second term lies in $z^Nf_0(A\otimes U) \subset M$, we also get $z^{N+2}\frac{\del f_0}{\del x} \otimes v_0 \in M$.
This shows that $z^{N'}\frac{\del f_0}{\del x} \in I$ for some $N'$, and by symmetry we also get $z^{N'}\frac{\del f_0}{\del y} \in I$.
Since $\nu_P(f_0)$ is finite there is some product $d$ of derivations with $d(f_0)(P) \neq 0$.
So acting repeatedly with elements as above we eventually obtain $z^K d(f_0)\in M$ and $z^K d(f_0)$ is nonzero at $P$. By the Nullstellensatz this means that $z^K \in I$ for some $K$, which in turn means that $M$ is dense. 
\end{proof}

%Finally, we need to check that density is preserved when taking direct sums of simple modules. Suppose $U$ and $U'$ are simple, and that $M \subset %A_{(z)} \otimes U$ and $N \subset A_{(z)} \otimes U'$ are both dense submodules. This means that $z^{-K}M \supset A\otimes U$ and $z^{-K}n \supset A
%\otimes U$  for sufficiently large $K$. But then also $z^K(M\oplus N)=z^KM\oplus Z^KN \supset (A\otimes U)\oplus (A\otimes U')=A\otimes (U\oplus U')$, so %$M\oplus N$ is indeed dense in $A_{(z)}\otimes(U\oplus U')$.

\begin{cor}
When $U$ is an irreducible $\fr{gl}_s$-module there exists at most one simple $A\D_{\Sp^2}$-submodule of $A_{(z)}\otimes U$.
\end{cor}
\begin{proof}
Let $M$ and $M'$ be simple submodules in $A_{(z)} \otimes U$. By Proposition~\ref{density} both modules are dense, so they both contain $z^N A\otimes U$ for large enough $N$. Thus $M \cap M'$ is a nonzero submodule of both $M$ and $M'$ so by simplicity we must have $M=M'$.
\end{proof}

\section{Tensor modules}
Consider two charts $N(h)$ and $N(\tilde{h})$ in our atlas for $X$.
Let $t_1, \ldots, t_s$ be chart parameters for $N(h)$ and let $\tilde{t}_1, \ldots, \tilde{t}_s$ be chart parameters for $N(\tilde{h})$.

Let $V$ be the natural $\fr{gl}_s$-module with basis $\{e_1, \ldots, e_s\}$.
Define an $A_{(\tilde{h})}$-linear map $C: A_{(h,\tilde{h})}\otimes V \ra A_{(h,\tilde{h})}\otimes V$ by \[Ce_i=\sum_{j=1}^s \frac{\del t_i}{\del \tilde{t}_j}e_j.\] 
Note that $C$ is invertible.

From now on we shall understand all $\GL_s$-modules (resp. $\fr{gl}_s$-modules) as $\GL(V)$-modules (resp. $\fr{gl}(V)$-modules).

For a finite-dimensional rational $\GL_s$-module $U$, write $\rho:\GL_s \ra \GL(U,k)$ for the corresponding representation. Then we may consider $\rho(C)$ as an element of $\GL(U,A_{(h)})$. We can also consider it as a map $A_{(h,\tilde{h})} \otimes U \ra A_{(h, \tilde{h})} \otimes U$.\\

Denote by $\fr{T}$ the full subcategory of $A\D$-Mod consisting of those objects $M$ that satisfy
\begin{itemize}
	\item $M$ is finitely generated as an $A$-module,	
	\item For each chart $N(h)$ in our fixed atlas there exists an injective $A\D$-module homomorphisms $\varphi_{h}: M \ra A_{(h)}\otimes U$,
	\item The following diagram commutes for each pair of charts $(N(h),N(\tilde{h}))$:
	\[\xymatrix@C=1.5cm@R=0.7cm{
        		& A_{(h)}\otimes U \ar@{^{(}->}[r] &  A_{(h,\tilde{h})}\otimes U \ar[dd]^{\rho(C)}\\
        M \ar@{^{(}->}[ur]^{\varphi_h} \ar@{^{(}->}[dr]_{\varphi_{\tilde{h}}}\\ 
                & A_{(\tilde{h})}\otimes U \ar@{^{(}->}[r] &  A_{(h,\tilde{h})}\otimes U\\}\]
\end{itemize}

The objects of the category $\fr{T}$ will be called tensor modules.

We point out that the modules $\Omega$ and $\D$ are in fact tensor modules. For $\Omega$, the natural $\fr{gl}_s$-module appearing in its construction is identified with $V$ via $e_i \leftrightarrow dt_i$. Then the transformation $C$ corresponds to the usual change of variables formula for the differentials. Clearly all compatibility conditions are satisfied and $\Omega$ is a tensor module.

The $\fr{gl}_s$-module corresponding to $\D$ is the dual of the natural module $(V^{*},\rho^{*})$. It can easily be checked that $\rho^{*}(C)$ corresponds to the change of variables formula for partial derivatives.

The category of tensor modules is also closed under taking tensor products and duals,
in particular we have $\Omega \simeq \D^{\circ}$.

Finally, for the sphere the rank $1$ modules $A^{\alpha}$ (for $\alpha \in 2\bb{Z}$) are also tensor modules with $\rho(C)=\det(C)^{\frac{\alpha}{2}}$.

\begin{thm}
Let $M \in \fr{T}$ be a tensor module on $\Sp^2$ corresponding to a simple rational $\GL_2$-module $U$. Then $M$ is a simple $A\D$-module. 
\end{thm}
\begin{proof}
Let $M'$ be a nonzero submodule of $M$ and define
\[I=\{f \in A | f M \subset M'\}.\]
Then $I$ is an ideal and it does not depend on the chart we use. Since $M$ is bounded, Proposition~\ref{density} implies that $z^N \in I$. By symmetry we also have $x^N \in I$ and $y^N \in I$ for some large enough $N$. But then the set of common zeros $V(I) \subset X$ is empty, and Hilbert's weak Nullstellensatz gives $1\in I$. In view of the definition of $I$ this says that $M=M'$.
\end{proof}

\begin{thm}
\label{subgen}
When $\frac{m-\alpha}{2}\in \bb{Z}$ the vector
\[w_m:=\sum_{i=0}^{m} z^{-\frac{\alpha+m}{2}}x^{m-i}y^{i}\otimes v_i\]
generates a bounded $A\D$-submodule inside $A_{(z)}\otimes U_{m}^{\alpha}$.
On the other hand, when $\frac{m-\alpha}{2}\not\in \bb{Z}$, the module $A_{(z)}\otimes U_{m}^{\alpha}$ contains no bounded submodules.
\end{thm}
\begin{proof}
We first prove the second part. Let $M$ be a nonzero bounded submodule of $A_{(z)}\otimes U^{\alpha}_{m}$.
Since $M$ is dense, it contains a vector of form $z^N\otimes v_0$ for large enough $N$. We may pick such a vector with minimal $N$.

We now compute
\[\Delta_{2,3}(z^N \otimes v_0) = -z^{N-1}y\left(N+\tfrac{1}{2}(\alpha-m)\right)\otimes v_0.\]
By the minimality of $N$ we must have $N+\frac{1}{2}(\alpha-m)=0$ which since $N$ is an integer means that $\frac{m-\alpha}{2}\in \bb{Z}$.

To prove the first statement, assume that $\frac{m-\alpha}{2}\in\bb{Z}$. 
First note that $w_m$ gives a correct generator for the cases discussed above: vector fields $\D$, $1$-forms $\Omega$, and all rank $1$-modules $A^{\alpha}$.
We now proceed by induction on $m$.
Dropping the tensor signs we may write $w_m$ as
\[w_m=z^{-\frac{\alpha+m}{2}}\sum_{i=0}^{m}{m\choose i}x^{m-i}y^{i}X^iY^{m-i} = z^{-\frac{\alpha+m}{2}}(xY+yX)^m.\]
Consider three $\fr{sl}_2$-submodules $U_{m}$, $U_{n}$, and $U_{m+n}$ of ${\bf k}[X,Y]$. The map $\varphi: U_{m} \otimes U_{n}\ra U_{m+n}$ given by multiplication in ${\bf k}[X,Y]$, ($f\otimes g \mapsto fg$), is a surjective homomorphism of $\fr{sl}_2$-modules. Introducing an action of the identity gives a corresponding surjective homomorphism of $\fr{gl}_2$-modules: $U_{m}^{\alpha} \otimes U_{n}^{\beta} \ra U_{m+n}^{\alpha+\beta}$. This morphism can now be further extended to a surjective homomorphism of $A\D$-modules:
\[\varphi: (A_{(z)}\otimes U_{m}^{\alpha})\otimes (A_{(z)}\otimes U_{n}^{\beta}) \ra A_{(z)}\otimes U_{m+n}^{\alpha + \beta}.\]
By the inductive assumption, $w_m$ and $w_n$ generate bounded submodules in $A_{(z)}\otimes U_{m}^{\alpha}$ and $A_{(z)}\otimes U_{n}^{\beta}$ respectively, so $w_m \otimes w_n$ generates a bounded submodule in $(A_{(z)}\otimes U_{m}^{\alpha}) \otimes (A_{(z)}\otimes U_{n}^{\beta})$.

Applying our multiplication map we conclude that $\varphi(w_m \otimes w_n)$ generates a bounded submodule in $A_{(z)}\otimes U_{m+n}^{\alpha + \beta}$.
 But 
\begin{align*}
\varphi(w_m \otimes w_n) &= \varphi\left(z^{-\frac{\alpha+m}{2}}(xY+yX)^m \otimes z^{-\frac{\beta+n}{2}}(xY+yX)^n\right)\\
&=z^{-\frac{(\alpha+\beta)+(m+n)}{2}}(xY+yX)^{m+n}=w_{m+n}.
\end{align*}
This concludes the proof.
\end{proof}

%Finally, for arbitrary $M$, we have $U_{m}^{2\alpha} \subset A^{2\alpha-m}\otimes \Omega^{\otimes m}$ so as above, it has a submodule of finite rank %generated by $w_m$.

\section{Tensor product decomposition}
Note in $A_{(h)}\otimes_A A_{(h)}$ we have 
\[h^{-1}\otimes 1 = h^{-1}\otimes hh^{-1}=h^{-1}h\otimes h^{-1} = 1\otimes h^{-1}.\]
This shows that $A_{(h)}\otimes_A A_{(h)} \simeq A_{(h)}$.
Therefore, from the classical decomposition of $\fr{gl}_2$-modules it follows that
\[(A_{(h)}\otimes U_{m}^{\alpha}) \otimes_{A} (A_{(h)}\otimes U_{n}^{\beta})=A_{(h)}\otimes(U_{m}^{\alpha} \otimes U_{n}^{\beta})
= \bigoplus_{i=0}^{n} A_{(h)}\otimes U^{\alpha + \beta}_{m+n-2i}.\]

We shall show that our $A\D$-modules which appear as submodules in $A_{(h)} \otimes U$ respect this decomposition.

\begin{thm}
In the category of tensor modules on the sphere, tensor products of simple modules decompose as a direct sums of simple tensor modules.
\end{thm}
\begin{proof}
Let $M$ and $N$ be simple tensor modules embedded as $M \subset A_{(z)}\otimes U_{m}^{\alpha}$ and $N \subset A_{(z)}\otimes U_{n}^{\beta}$ with $m\geq n$.
Then \[M\otimes_A N \subset (A_{(z)}\otimes U_{m}^{\alpha}) \otimes_{A} (A_{(z)}\otimes U_{n}^{\beta})=A_{(z)}\otimes(U_{m}^{\alpha} \otimes U_{n}^{\beta})
= \bigoplus_{i=0}^{n} A_{(z)}\otimes U^{\alpha + \beta}_{m+n-2i}.\]
Write $\pi_k$ for the projection onto the $k$-th direct summand: \[\pi_k:A_{(z)}\otimes(U_{m}^{\alpha} \otimes U_{n}^{\beta}) \ra  A_{(z)}\otimes U^{\alpha + \beta}_{m+n-2k}.\] Then $\mathrm{id}_{A_{(z)}\otimes(U_{m}^{\alpha} \otimes U_{n}^{\beta})} = \bigoplus_{k=0}^{n} \pi_k$. We shall show that this decomposition still holds when restricted to the subspace $M\otimes_A N$. For this it suffices to check that $\pi_k(M\otimes N) \subset M\otimes N$.
Let $v\in M\otimes_A N$. By the density of $M$ and $N$, we see that $M\otimes N$ is dense in $\bigoplus_{i=0}^{n} A_{(z)}\otimes U^{\alpha + \beta}_{m+n-2i}$ as well. This implies that $\pi_k(M\otimes N)$ is nonzero for each $k$, so $\pi_k(M\otimes N)$ is a nonzero bounded submodule of $A_{(z)}\otimes U_{m+n-2k}^{\alpha+\beta}$, and by previous results $\pi_k(M\otimes N)$ is dense and simple. 

Now for arbitrary $j\geq 0$ we have $z^jv = \sum_{k}\pi_k(z^jv)$, 
and for $j$ large enough we get $\pi_k(z^jv)\in M\otimes N$ by the density of $M\otimes N$.

Next, by the simplicity of $\pi_k(M\otimes N)$, we have \[\pi_k(v) \in A\#U(\D) \cdot \pi_k(z^jv) \subset M\otimes N.\]
Thus we have shown that $id_{M\otimes N} =\bigoplus \pi_k|_{M\otimes N}$ and hence 
\[M\otimes N = \bigoplus \pi_k(M\otimes N).\]

%Next, by the simplicity of $\pi_k(M\otimes N)$, we have $v \in A\#U(\D) \cdot z^jv \subset M\otimes N$.
%Thus we have shown that $id_{M\otimes N} =\bigoplus \pi_k|_{M\otimes N}$ and hence 
%\[M\otimes N = \bigoplus \pi_k(M\otimes N)\]
\end{proof}

\begin{cor}
The category of finite-dimensional rational $\GL_2$-modules is equivalent to a full subcategory of the category $\fr{T}$ of tensor modules on the sphere.
Moreover, this subcategory is generated by $A^{-2}$ and $\Omega$ as a monoidal abelian category.
\end{cor}
\begin{proof}
The equivalence is provided by the functor $F: U \mapsto \text{soc}(A_{(z)}\otimes U)$. This is a bijection between simple finite-dimensional rational $\GL_2$-modules and simple tensor modules.

For the second statement we note that the tensor products of $A\D$-modules of rank $1$ is given by
\[A^{\alpha} \otimes A^{\beta} \simeq A^{\alpha+\beta}.\]
In particular, we get $A^{2k}=(A^{2})^{\otimes k}$ and $A^{-2k}=(A^{-2})^{\otimes k}$ for any $k\in \bb{N}$. We also have
\[\D \otimes_A A^{2} \simeq \Omega \qquad \text{ and } \qquad \Omega \otimes_A A^{-2} \simeq \D.\]

Let $M$ be a bounded submodule of $A_{(z)}\otimes U_{m}^{\alpha}$. Then $\frac{m-\alpha}{2}$ is an integer by Lemma~\ref{subgen}. Now let $V=U_{1}^{1}$ be the natural $\fr{gl}_2$-module. It is well known that $U_m \subset V^{\otimes m}$ as $\fr{sl_2}$-modules, so we get $U_{m}^{\alpha} \subset A^{\alpha-m}\otimes V^{\otimes n}$. This in turn implies that $M \subset A^{\alpha-m} \otimes_A \Omega^{\otimes m}$.

Thus the category $\fr{T}$ is generated by $A^{2}$, $A^{-2}$, and $\Omega$ as a monoidal abelian category. However, $A^{2}$ is a direct summand of $\Omega \otimes \Omega$ so it may be dropped as a generator.
\end{proof}

%Next, any $GL_2$-module where identity acts as $1$ is a direct summand of a tensor power of the natural module $N$, so an arbitrary rational $GL_2$ module %can be obtained from such a module tensored with a $1$-dimensional module where identity acts as an integer. But $\text{soc}(A_{(z)} \otimes N) \simeq 
%\Omega$. This shows that for an arbitrary tensor module $M$ we have $M \subset (A^{\pm 1})^{\otimes \alpha}\otimes (\Omega^{\otimes k})$, so that as a %monoidal abelian category, $\fr{T}$ is generated by $\Omega, A^{1},$ and $A^{-1}$. However, $A^{1}$ is a direct summand of $\Omega \otimes \Omega$ so it %may be dropped as a generator.

\end{document}